\documentclass[12pt,a4paper]{article}
\usepackage{mathrsfs,amsmath,amsfonts,amssymb,amsthm,graphicx,color}
\newcommand{\p}{\partial}

\renewcommand{\leq}{\leqslant}
\renewcommand{\geq}{\geqslant}

\newcommand{\rot}{\operatorname{\textbf{rot}}}
\newcommand{\curl}{\operatorname{\textbf{curl}}}
\newtheorem{theorem}{{\bf Theorem}}[section]
\newtheorem{definition}[theorem]{Definition}
\newtheorem{corollary}[theorem]{{\bf Corollary}}
\newtheorem{lemma}[theorem]{{\bf Lemma}}
\newtheorem{prop}[theorem]{{\bf Proposition}}
\newtheorem{remark}[theorem]{{\bf Remark}}
\begin{document}
\title{\textbf{Unique Continuation and Observability Estimates for 2-D Stokes Equations with the  Navier Slip  Boundary Condition}}
\date{}
\author{ Yuning Liu\thanks{Faculty of Mathematics,
University of Regensburg, D-93053,
Regensburg, Germany. (liuyuning.math@gmail.com). The author was partially supported by the University of Regensburg.}\quad\quad
Can Zhang\thanks{
School of Mathematics and Statistics, Wuhan University, Wuhan,
430072, China. (zhangcansx@163.com). The author was partially supported by the National Natural Science Foundation of China under grants
11161130003 and 11171264.} }

\maketitle
\abstract{This paper presents  a unique continuation estimate for 2-D Stokes equations with the Naiver slip boundary condition in a bounded and simply connected domain. Consequently, an observability estimate for this equation from a subset of positive measure in time follows from the aforementioned unique continuation estimate and the new strategy developed in \cite{unpublished1}. Several applications of the above-mentioned observability estimate to  control problems of the Stokes equations are given.}\\

\noindent{\bf Keywords.} Stokes equations, observability estimate, unique continuation estimate, bang-bang property\\

\noindent\textbf{AMS Subject Classifications.} 93B07, 35Q30,
35B60

\section{Introduction and Main Results}
$\;\;\;\;$Let $T>0$ and $\Omega\subset\mathbb{R}^{2}$ be a bounded and simply connected domain with  a $C^3$ boundary $\partial\Omega$.
Let $\mathbf{n}$ be the unit exterior normal vector to
$\partial\Omega$ and $\mathbf{\tau}$ be the unit tangent vector
to $\partial\Omega$ such that $(\mathbf{n},\mathbf{\tau})$
is positively oriented. Consider the following Stokes equations:
\begin{equation}\label{stokes}
\begin{cases}
\mathbf{u}_t-\Delta \mathbf{u}-\nabla p=0\;\;&\text{in}\;\;\Omega\times(0,T),\\
\nabla\cdot \mathbf{u}=0\;\;&\text{in}\;\;\Omega\times(0,T),\\
\rot\mathbf{u}=0,\;\;\mathbf{u}\cdot \mathbf{n}=0\;\;&\text{on}\;\;\partial\Omega\times(0,T).
\end{cases}
\end{equation}
The boundary condition  in \eqref{stokes} is a  special Navier slip boundary condition. In general, the Navier slip boundary condition reads (see for instance \cite{cor1})
\begin{equation}\label{naviercondition}
  \left\{\begin{array}{ll}
       \mathbf{u}\cdot\mathbf{n}=0,\\
       \bar{\sigma}\mathbf{u}\cdot\mathbf{\tau}+(1-\bar{\sigma})n_i\left(\frac{\p u_i}{\p x_j}+\frac{\p u_j}{\p x_i}\right)\tau_j=0,
     \end{array}
  \right.
  \end{equation}
  where $\bar{\sigma}$ is a constant in $[0,1)$. Here we used the Einstein's summation convention.

For the unique continuation of Stokes  equations, there have been many literatures. Here, we would like to mention the classical reference \cite{prolonge} (see also \cite{MR1900552}), where some qualitative unique continuation results are provided.  The controllability of Stokes equations or Navier-Stokes equations, with either the Dirichlet boundary condition or the Navier-slip boundary condition are investigated in a large number of references. For Stokes equations with Dirichlet boundary condition, we refer to \cite{MR2503028}, \cite{MR2103189} and the references therein. In  \cite{FursIm}, the controllability of the 2-D linearized Navier-Stokes equations with the same Navier-slip boundary condition as that in (\ref{stokes})  is systematically studied with the aid of the global Carleman inequality. In \cite{cor1}, Jean-Michel Coron studied the approximate controllability of the 2-D Navier-Stokes equations with the general Navier-slip boundary conditions  \eqref{naviercondition}.

  However, to our best knowledge,  very limited works are concerned with the quantitative unique continuation for Navier-Stokes equations. Here, we would like to mention the paper  \cite{kuc}
 where   an upper bound  was given for the size
of the nodal set of the vorticity for a solution (but not the nodal set of a solution) of the 2-D
periodic Navier-Stokes equations.

  The main purpose of this paper is to present the quantitative unique continuation for Equations (\ref{stokes}).
   It should be mentioned that the current study is  greatly motivated by  recent papers \cite{MR2652187}, \cite{unpublished1} and \cite{wangzhang}, where some kinds of unique continuation estimates for heat (or parabolic) equations were established. To present the main result of this paper, we begin by introducing the following notations:
\begin{equation*}\label{vectorvalued}
L^2_{\sigma}(\Omega)=\{\mathbf{u}\in L^2(\Omega ; \mathbb{R}^2):~\nabla\cdot\mathbf{u}=0,~\mathbf{u}
\cdot\mathbf{n}|_{\p\Omega}=0\};
\end{equation*}
\begin{equation*}\label{vectorvalued1}
H^1_{\sigma}(\Omega)=\{\mathbf{u}\in H^1(\Omega;\mathbb{R}^2):~\nabla\cdot\mathbf{u}=0,~\mathbf{u}
\cdot\mathbf{n}|_{\p\Omega}=0\}.
\end{equation*}
When $O$ is a subset in $\mathbb{R}^2$, we will write accordingly $L^2(O)$ and  $H^1(O)$
 for  $L^2(O;\mathbf{\mathbb{R}}^2)$ and $H^1(O;\mathbf{\mathbb{R}}^2)$, if there is no chance to make any confusion. We will denote by $(\cdot,\cdot)$ the usual inner product in $L^2(\Omega)$ and by $\chi_{E}$ the characteristic function of the subset $E$. In this paper, $N(\cdot)$ stands for a positive constant depending on what are enclosed in the brackets. It maybe vary in different contexts.

The  main results of this paper are included in the following theorem:
\begin{theorem}\label{uniquecontinuation}
Let $T>0$ and $\Omega\subset\mathbb{R}^{2}$ be a bounded and simply connected domain with  a $C^3$ boundary $\partial\Omega$.
 Let $\omega\subset\Omega$ be a nonempty open subset.
Then,

\noindent$(i)$ There are $N=N(\Omega,\omega)$ and $\alpha=\alpha(\Omega,\omega)$, with $\alpha\in(0,1)$, such that when $0\leq t_1<t_2\leq T$ and $\mathbf{u}_0\in L^2_{\sigma}(\Omega)$,
 the solution $\mathbf{u}$ to Equations \eqref{stokes}, with the initial condition $\mathbf{u}(0)=\mathbf{u}_0$, verifies
\begin{equation}\label{c*}
\|\mathbf{u}(t_2)\|_{L^2(\Omega)}\leq\left
(Ne^{\frac{N}{t_2-t_1}}\|\mathbf{u}(t_2)\|_{L^2(\omega)}
\right)^{\alpha}\|\mathbf{u}(t_1)\|
^{1-\alpha}_{L^2(\Omega)}.
\end{equation}
\noindent $(ii)$
For each subset $E\subset(0,T)$  of positive measure, there exists $N=N(\Omega,\omega,E,T)$ such that when
$\mathbf{u}_0\in L^2_{\sigma}(\Omega)$,
 the solution $\mathbf{u}$ to Equations \eqref{stokes}, with the initial condition $\mathbf{u}(0)=\mathbf{u}_0$, satisfies
\begin{equation}\label{c88}
  \|\mathbf{u}(T)\|_{L^2(\Omega)}\leq N\int^T_{0}\chi_{E}\|\mathbf{u}(t)\|
  _{L^2(\omega)}\,dt.
\end{equation}
\end{theorem}

Our strategy to prove the estimate \eqref{c*} is as follows: $(a)$ We transform Equations \eqref{stokes} into  one of parabolic
type  via the method of stream function; $(b)$ We
apply the estimate established in \cite{wangzhang} (see also \cite{MR2652187} and \cite{unpublished1} for the case where $\Omega$ is convex), together with a type of the Sobolev interpolation inequality and some properties of heat equations,
to get a unique continuation estimate
for the stream functions; $(c)$ We  pull  the unique continuation estimate for
stream functions back to the desired  estimate \eqref{c*}.

Three remarks are in order:

\begin{itemize}
\item  The estimate \eqref{c*} is not a trivial consequence of the corresponding unique continuation estimate
for heat equations built up in \cite{wangzhang} (see also \cite{MR2652187} and \cite{unpublished1}).

\item In Step $(c)$, it will be used that  $\Omega$ is simply connected.

\item Based on the estimate \eqref{c*}, the observability estimate \eqref{c88} follows from the new strategy developed in \cite{unpublished1}  at once.
The estimate \eqref{c88} leads to the null-controllability of
Equations \eqref{stokes} with controls restricted over $\omega\times E$.  Since $E$ is a measurable subset in time, such null-controllability for Stokes equations  seems to be new.
\end{itemize}

The rest of this paper is organized as follows: Section 2
presents some preliminaries; Section 3 proves Theorem~\ref{uniquecontinuation}; Section 4  provides  some applications of Theorem~\ref{uniquecontinuation} in the control theory of Stokes equations and Section 5, i.e.,  Appendix contains the proof of some elementary results used in this study.

\section{Some Preliminaries}
\setcounter{equation}{0}
$\;\;\;\;$This section is devoted to review some classical results on the decomposition of two-dimensional vector fields and to prove the well-posedness of Equations \eqref{stokes}. A comprehensive discussion of the decomposition of 2-D vector fields can be found in  \cite[Appendix I, pp. 458-469]{TemamNavier} or \cite[pp. 18-56]{femstokes} .

For each $\psi\in H^1(\Omega)$ and each $\mathbf{u}=(u_1,u_2)\in H^1(\Omega)$,  define
\begin{equation*}\label{rotdefinition}
  \curl\psi=(\p_2\psi,-\p_1\psi),\;\;
  \rot\mathbf{u}=\p_1u_2-\p_2u_1 .
\end{equation*}
It can be easily verified that
\begin{equation}\label{rotformula}
\begin{split}
&\curl\rot\mathbf{u}=-\Delta\mathbf{u},\;\;\text{when}\;\;\mathbf{u}\in H^1_\sigma(\Omega)\cap H^2(\Omega);\\
&\curl\psi\in L^2_{\sigma}(\Omega),\;\;\rot\curl\psi=-\Delta\psi,\;\;
\text{when}\;\;\psi\in H^1_0(\Omega)\cap H^2(\Omega).\\
\end{split}
\end{equation}
The following lemma gives a kind of  Green formula connected with  operators  $\rot$ and $\curl$:
\begin{lemma}\label{lemmacan2}
For any $\mathbf{u}\in H^2(\Omega)$ and $\mathbf{v}\in H^1(\Omega)$,
\begin{equation*}\label{green}
  \int_{\Omega}(\curl\rot\mathbf{u})\cdot\mathbf{v}\,dx=
  -\int_{\p\Omega}(\rot\mathbf{u})(\mathbf{v}\cdot\tau )\, ds+\int_{\Omega}(\rot\mathbf{u})(\rot\mathbf{v})\,dx.
\end{equation*}
\end{lemma}
\begin{proof}
Set $\phi=\rot\mathbf{u}$. From the standard Green formula,
\begin{equation*}
\begin{split}
   \int_{\Omega}(\curl\phi)\cdot\mathbf{v}\,dx&=\int_{\Omega}(\p_2\phi
    v_1-\p_1\phi v_2)\,dx\\
    &=\int_{\p\Omega}\phi(v_1n_2-v_2n_1)\,ds-\int_{\Omega}
    \phi(\p_2v_1-\p_1v_2n_1)\,dx\\
    &=-\int_{\p\Omega}\phi(\mathbf{v}\cdot\tau)\,ds+
    \int_{\Omega}\phi\rot\mathbf{v}\,dx,
\end{split}
\end{equation*}
which leads to the desired equality.
\end{proof}

The next lemma concerns with the well-poseness of the parabolic-type equation satisfied by  stream functions. Such well-posedness for a similar equation as Equation (\ref{stream}) in the next lemma
was built up in \cite[Th\'{e}or\`{e}me 6.10, pp. 88]{quelques}.  For the sake of completion,  we give its proof in Appendix.

\begin{lemma}\label{biharmoniclemma}
Let $\Omega\subset\mathbb{R}^{2}$ be a bounded domain with a $C^3$ boundary $\partial\Omega$.
Then, for each $\psi_0\in H^1_0(\Omega)$,  there exists a unique  solution $\psi$, with
\begin{equation}\label{regular}
\begin{split}
&\psi\in C([0,T]; H^1_0(\Omega))\cap C^1((0,T]; H^1_0(\Omega))\cap C((0,T]; H^3(\Omega))\\ &\;\;\text{and}\; \;\Delta\psi\in C((0,T]; H^1_0(\Omega)),
\end{split}
\end{equation}
to the equation
 \begin{equation} \label{stream}
  \begin{cases}
\Delta\psi_t-\Delta^2\psi=0\;\;\;\;&\text{in}\;\;  \Omega\times(0,T),\\
\Delta\psi=0,\;\;\psi=0\;\;\;\;\;&\text{on}\;\;\partial\Omega\times(0,T),\\
\psi(\cdot,0)=\psi_0(\cdot)\;\;\;\;&\text{in}\;\;\Omega.
   \end{cases}
  \end{equation}
\end{lemma}

\begin{lemma}\label{rottheorem}
  Let $\Omega$ be a bounded and simply connected  domain with a $C^2$ boundary $\partial\Omega$. Then, for each $\mathbf{u}\in L^2_{\sigma}(\Omega)$, there exists a  unique stream function $\psi\in H^1_0(\Omega)$ such that $\curl\psi=\mathbf{u}$.
\end{lemma}
The proof of  lemma~\ref{rottheorem} is essentially contained in
 \cite[Theorem 3.1, pp. 37-40]{femstokes}, where the curl equation $\curl\psi=\mathbf{u}$ was studied for the case that $\Omega$ is multi-connected. The corresponding  result there reads: for each $\mathbf{u}\in L^2_{\sigma}(\Omega)$, the curl  equation has a solution in $H^1(\Omega)$, which is a constant on each connected component of $\p\Omega$. These constants may be different in different connected components. Hence, it may happen that  any solution of the curl equation is not in  $H^1_0(\Omega)$ in that case.   For the sake of convenience, we provide a proof for Lemma~\ref{rottheorem} in Appendix  of this paper. \\

\begin{prop}\label{wellpose}
Let $\Omega\subset\mathbb{R}^{2}$ be a bounded and simply connected domain with a $C^3$ boundary $\partial\Omega$.
Then, for each  $\mathbf{u}_0\in L^2_{\sigma}(\Omega)$, Equations \eqref{stokes}, with the initial condition $\mathbf{u}(\cdot, 0)=\mathbf{u}_0$, has a unique solution
\begin{equation}\label{regular2}
    \mathbf{u}\in C([0,T];L^2_{\sigma}(\Omega))\cap  C^1((0,T];L^2_{\sigma}(\Omega)) \cap  C((0,T];H^2(\Omega)\cap H^1_{\sigma}(\Omega)),
\end{equation}
  for some $p\in L_{loc}^2(0,T;L^2(\Omega))$. Moreover, $\mathbf{u}\in L^2(0,T;H^1_{\sigma}(\Omega))$ and
\begin{equation}\label{estimatecz}
\max_{s\in[0,T]}\|\mathbf{u}(s)\|^2_{L^2(\Omega)}+
\int^{T}_0\|\mathbf{u}(t)\|^2_{H^1(\Omega)}\,dt
\leq N(\Omega)\|\mathbf{u}_0\|^2_{L^2(\Omega)}.
\end{equation}
\end{prop}
\begin{proof}
   According to Lemma $\ref{rottheorem}$, there is a unique $\psi_0\in H^1_0(\Omega)$ such that $\curl \psi_0=\mathbf{u}_0$. Then, by Lemma $\ref{biharmoniclemma}$,  Equation \eqref{stream}, with the aforementioned $\psi_0$, has a unique solution $\psi$ verifying \eqref{regular}.
We claim that the vector field $\mathbf{u}:=\curl\psi$ satisfies Equations \eqref{stokes}, as well as \eqref{regular2}. In fact,
it can be checked readily that $\nabla\cdot\mathbf{u}=\nabla\cdot\curl\psi=0$ and that
\begin{equation*}
\begin{split}
&\rot\mathbf{u}=\rot\curl\psi=-\Delta\psi=0,\;\;\text{on}\;\;\p\Omega
\times(0,T);\\
&\mathbf{u}\cdot\mathbf{n}=\curl\psi\cdot\mathbf{n}=\frac{\p\psi}
 {\p\mathbf{\tau}}=0,\;\;\text{on}\;\;\p\Omega\times(0,T).
\end{split}
\end{equation*}
Also, by Equation \eqref{stream},
\begin{equation*}
  \rot(\mathbf{u}_t-\Delta \mathbf{u})
    =\rot(\curl\psi_t-\Delta\curl\psi)=0.
\end{equation*}
Because $\Omega$ is simply connected,  for a.e. $t\in(0,T)$, there exists a unique function $p(t)\in H^1(\Omega)$ up to a constant such that (see for instance \cite[Theorem 2.9, pp. 31]{femstokes})
\begin{equation*}
  \mathbf{u}_t-\Delta \mathbf{u}=\nabla p.
\end{equation*}
From the Poincar\'{e} inequality, it follows that there exists
$p\in L_{loc}^2(0,T;L^2(\Omega))$ provided that $\int_{\Omega}p(t,x)\,dx=0$ for a.e. $t\in(0,T)$.

To justify  \eqref{estimatecz},  we  multiply the first equation of \eqref{stokes} by $\mathbf{u}$, and then integrate it from $\varepsilon>0$ (sufficiently small) to $s\in (0,T]$. Now,  Lemma~\ref{lemmacan2}  leads to
\begin{equation}\label{2.1}
\begin{split}
&\|\mathbf{u}(s)\|^2_{L^2(\Omega)}+2\int^{s}_\epsilon
\|\rot\mathbf{u}(t)\|^2_{L^2(\Omega)}\,dt
=\|\mathbf{u}(\epsilon)\|^2_{L^2(\Omega)}.
\end{split}
\end{equation}
Sending $\epsilon\rightarrow 0$ in \eqref{2.1}, we see that
\begin{equation*}\label{estimate}
\|\mathbf{u}(s)\|^2_{L^2(\Omega)}+
2\int^{s}_0\|\rot\mathbf{u}(t)\|^2_{L^2(\Omega)}\,dt
=\|\mathbf{u}_0\|^2_{L^2(\Omega)}.
\end{equation*}
This, together with the simply connectedness of $\Omega$ and the decomposition theorem (see, e.g., \cite[Remark 3.5, pp.\,45]{femstokes} or \cite[Lemma 1.6, pp.\,465]{TemamNavier}), indicates the estimate \eqref{estimatecz}, from which, the uniqueness follows at once.
\end{proof}

\section{Unique Continuation Estimates}
\setcounter{equation}{0}
$\;\;\;\;$This section is devoted to prove Theorem~\ref{uniquecontinuation}.
We first establish an estimate for the gradient of the stream function, and then present the proof of estimate (\ref{c*})
and (\ref{c88}), where
the simply connectedness of the domain  is used. In what follows, we will denote by $\psi$  a  solution to the equation:
\begin{equation} \label{streamWGS}
  \begin{cases}
\Delta\psi_t-\Delta^2\psi=0\;\;\;\;&\text{in}\;\;  \Omega\times(0,T),\\
\Delta\psi=0,\;\;\psi=0\;\;\;\;\;&\text{on}\;\;\partial\Omega\times(0,T),\\
\psi(0)\in H^1_0(\Omega).
   \end{cases}
  \end{equation}

\begin{lemma}\label{prop1}For any $0\leq s<t\leq T$,
$\|\nabla\psi(t)\|_{L^2(\Omega)}\leq\|\nabla\psi(s)\|_{L^2(\Omega)}$. Moreover, $\nabla\psi(t)=0$ for all $t\in[0,T]$ whenever $\nabla\psi(T)=0$.
\end{lemma}
\begin{proof}
First, set
$$e(t)=\int_{\Omega}|\nabla\psi(t)|^2\,dx,\;\;t\in[0,T].$$
Since $\Delta\psi=0$ and $\psi=0$ on $\partial\Omega\times(0,T)$, we have that for each $t>0$,
  \begin{equation*}
  \begin{split}
    &\dot e(t)=
    \int_{\Omega}2\nabla\psi\cdot\nabla\psi_t\,dx
    =2\int_{\p\Omega}\psi\frac{\p\psi_t}{\p n}\,ds  -2\int_{\Omega}\psi\Delta\psi_t\,dx\\
    &=-2\int_{\Omega}\psi\Delta(\Delta\psi)\,dx
    =-2\int_{\Omega}|\Delta\psi|^2dx-2\int_{\p\Omega}\left(\psi\frac{\p}{\p n}(\Delta\psi)-\Delta \psi\frac{\p\psi}{\p n}\right) \,ds\\
    &=-2\int_{\Omega}|\Delta\psi|^2\,dx.
    \end{split}
  \end{equation*}
Namely, it holds that
\begin{equation}\label{z1}
   \dot e(t)=-2\int_{\Omega}|\Delta\psi|^2\,dx.
\end{equation}

Next,
\begin{equation}\label{z2}
\begin{split}
 & \ddot e(t)
  =-4\int_{\Omega}\Delta\psi\Delta\psi_t \,dx=-4\int_{\Omega}\Delta\psi\Delta(\Delta\psi)\, dx\\
  &=-4\int_{\p\Omega}\Delta\psi\frac{\p}{\p n}(\Delta\psi)\,d\sigma+4\int_{\Omega}\nabla\Delta\psi\cdot\nabla\Delta\psi \,dx=4\int_{\Omega}|\nabla(\Delta\psi)|^2\,dx.
\end{split}
\end{equation}
Integrating by parts and  the Cauchy-Schwartz inequality lead to
 \begin{equation*}
   \int_{\Omega}|\Delta\psi|^2dx=-\int_{\Omega}\nabla(\Delta\psi)\cdot\nabla\psi dx \leq \left(\int_{\Omega}|\nabla\Delta\psi|^2dx\right)^{\frac 12}\left(\int_{\Omega}|\nabla\psi|^2dx\right)^{\frac 12}.
 \end{equation*}
 This, together with \eqref{z1} and \eqref{z2},
shows that
 \begin{equation}\label{convex}
  \ddot e(t)e(t)\geq(\dot e(t))^2,\;\;t\in(0,T).
 \end{equation}
If $e(t)=0$ for all $0\leq t\leq T$, we are done.
  Otherwise there exists a closed interval $[t_1,t_2]\subset [0,T]$ on which
  \begin{equation}\label{z4}
    e(t)>0 \quad \text{for all}\quad t\in[t_1, t_2);\;\;\mbox{and}\;\; \quad e(t_2)=0.
  \end{equation}
  Now, write $h(t):= \log e(t)$, $t_1\leq t< t_2$. Then, it follows from  \eqref{convex} that
  $$h''(t)\geq 0\;\;\mbox{ for all}\;\;t\in (t_1,t_2),
  $$
  \noindent i.e.,
  $h(t)$ is convex on $(t_1,t_2)$. Thus,
    $$h\left(\frac{t_1+t}{2}\right)\leq \frac{h(t_1)+h(t)}{2},\;\;t\in(t_1,t_2).$$
This implies that
\begin{equation*}
 e\left(\frac{t_1+t}{2}\right) \leq (e(t_1))^{\frac 12}(e(t))^{\frac 12},
    \;\;t\in(t_1,t_2).
\end{equation*}
 sending $t \rightarrow t_2$ in the above identity, we get
  $e\big(\frac{t_1+t_2}{2}\big)=0$, which contradicts  (\ref{z4}).
\end{proof}

\begin{lemma}\label{phungwang}
Let $\omega$ be a nonempty open subset of $\Omega$.
Then, there are  constants $N=N(\Omega,\omega)$ and $\alpha=\alpha(\Omega,\omega)\in(0,1)$,   such that when $0\leq t_1<t_2\leq T$,
  \begin{equation}\label{threeball}
    \|\nabla\psi(t_2)\|_{L^2(\Omega)}\leq\left(Ne^{\frac{N}{t_2-t_1}}\|\nabla
    \psi(t_2)\|_{L^2(\omega)}\right)
    ^{\alpha}\|\nabla\psi(t_1)\|
    ^{1-\alpha}_{L^2(\Omega)},
  \end{equation}
  for all solutions to Equation (\ref{streamWGS}).
 Consequently, $\psi\equiv0$ when $\nabla\psi(T,x)=0$ for a.e. $x\in\omega$.
  \end{lemma}
\begin{proof}
It suffices to prove the estimate \eqref{threeball} when $0< t_1<t_2\leq T$. Indeed, if it is the case, then, by  taking $t_1=\frac{t_2}{2}$
and noting that $\|\nabla\psi(t_1)\|_{L^2(\Omega)}\leq \|\nabla\psi(0)\|_{L^2(\Omega)}$ (see Lemma $\ref{prop1}$), we reach the estimate (\ref{threeball}) for $t_1=0$.

Because $\omega$ is a non-empty open set, there exists a ball
$B_r$, centered at a point $x_0\in\omega$ and of radius $r>0$,
such that $B_r\subset \omega$.
Since $\Delta\psi(\cdot)$ satisfies the heat equation with the zero Dirichlet boundary condition, it follows from \cite[Theorem 6]{wangzhang} (see also \cite[Proposition 2.1]{MR2652187} or \cite[Proposition 2.2]{unpublished1}) that   there
are constants $N=N(\Omega,B_r)$ and $\alpha=\alpha(\Omega,B_r)\in(0,1)$, such that
\begin{equation}\label{1.3}
  \|\Delta\psi(t_2)\|_{L^2(\Omega)}\leq \|\Delta\psi(t_3)\|_{L^2(\Omega)}^{\alpha}\left(Ne^{\frac{N}
  {t_2-t_3}}\|\Delta\psi(t_2)\|_{L^2(B_r)}\right)^{1-\alpha},
\end{equation}
when  $0< t_1<t_3<t_2\leq T$, and $\psi$ solves Equation (\ref{streamWGS}).

Because $\psi=0$ on $\partial\Omega\times(0,T)$, we obtain from the  regularity of elliptic equations that
\begin{equation}\label{1.4}
  \|\nabla\psi(t_2)\|_{L^2(\Omega)}\leq N(\Omega)\|\Delta\psi(t_2)\|_{L^2(\Omega)}.
\end{equation}
From the estimate:
\begin{equation*}
\|\Delta\psi(t_2)\|^2_{L^2(B_r)}\leq \|\nabla\psi(t_2)\|^2_{H^1(B_r)}
\end{equation*}
and the Sobolev interpolation inequality (see, for instance,
\cite[Theorem 5.2, pp.135]{af} or \cite[pp. 43-44]{lions}):
\begin{equation*}
\|\nabla\psi(t_2)\|_{H^1(B_r)}\leq N(B_r)\|\nabla\psi(t_2)\|^{\frac 12}_{H^2(B_r)}\|\nabla\psi(t_2)
\|^{\frac 12}_{L^2(B_r)},
\end{equation*}
it follows that
\begin{equation}\label{1.5}
 \|\Delta\psi(t_2)\|^2_{L^2(B_r)}\leq N(B_r)\|\nabla\psi(t_2)\|_{H^2(B_r)}\|\nabla\psi(t_2)
  \|_{L^2(B_r)}.
\end{equation}

On the other hand, since  $\Delta\psi=0$ on $\partial\Omega\times(0,T)$,    integrating by parts and the Cauchy-Schwartz inequality lead to
\begin{equation}\label{1.6}
  \|\Delta\psi(t_3)\|^2_{L^2(\Omega)}\leq \|\nabla\Delta\psi(t_3)\|_{L^2(\Omega)}\|\nabla\psi(t_3)
  \|_{L^2(\Omega)}.
\end{equation}
Combining inequalities \eqref{1.3}---\eqref{1.6}, we deduce that
\begin{equation}\label{1.9}
  \|\nabla\psi(t_2)\|_{L^2(\Omega)}\leq I_1^{\frac{\alpha}{2}}I_2^{\frac{\alpha}{2}}
  I_3^{\frac{1-\alpha}{2}}I_4^{\frac{1-\alpha}{2}},
\end{equation}
where
\begin{equation*}
  \left\{
  \begin{array}{ll}
I_1=\|\nabla\Delta\psi(t_3)\|_{L^2(\Omega)},\\
I_2=\|\nabla\psi(t_3)\|_{L^2(\Omega)},\\
I_3=\|\nabla\psi(t_2)\|_{H^2(B_r)},\\
I_4=Ne^{\frac{N}{t_2-t_3}}\|\nabla\psi(t_2)\|_{L^2(B_r)},\;\;
N=N(\Omega,B_r).
   \end{array}
  \right.
 \end{equation*}

Next, write $\{\lambda_i\}_{i\geq 1}$, with $0<\lambda_1<\lambda_2\leq \lambda_3\leq\cdots$, for the eigenvalues of the Laplace operator $-\Delta$ with zero Dirichlet boundary condition, and $\{e_i\}_{i\geq 1}$ for the corresponding set of $L^2(\Omega)$-normalized
eigenfunctions. Since $\Delta\psi$ satisfies the heat equation with
zero Dirichlet boundary condition, it holds that
\begin{equation*}\label{formula}
  \Delta\psi(t_3)=\sum_{i\geq1}e^{-\lambda_i(t_3-t_1)}(\Delta\psi(t_1)
  ,e_i)e_i.
\end{equation*}
Thus,
\begin{equation}\label{1.71}
\begin{split}
I^2_1&= \sum_{i\geq1}|(\Delta\psi(t_1),e_i)|^2
\lambda_ie^{-2\lambda_i(t_3-t_1)}\\
&\leq \sup_{i\geq 1}\left(\lambda_ie^{-\lambda_i(t_3-t_1)}\right)
\sum_{i\geq1}|(\Delta\psi(t_1),e_i)|^2e^{-\lambda_i(t_3-t_1)}\\
&\leq\frac{1}{t_3-t_1}\left\|\Delta\psi\left(\frac{t_1+t_3}{2}\right)
\right\|^2_{L^2(\Omega)}.\\
\end{split}
\end{equation}
Because $\psi$  verifies the heat equation with zero Dirichlet boundary condition, it stands that
\begin{equation*}\label{formula}
  \psi\left(\frac{t_1+t_3}{2}\right)=\sum_{i\geq1}e^{-\lambda_i(t_3-t_1)/2}
  (\psi(t_1),e_i)e_i.
\end{equation*}
Hence,
\begin{equation}\label{he1}
\begin{split}
\left\|\Delta\psi\left(\frac{t_1+t_3}{2}\right)
\right\|^2_{L^2(\Omega)}
&=\sum_{i\geq 1}\lambda_i^2 e^{-\lambda_i(t_3-t_1)}|(\psi(t_1),e_i)|^2\\
&\leq \sup_{i\geq1}\left(\lambda_ie^{-\lambda_i(t_3-t_1)}\right)
\sum_{i\geq 1}\lambda_i|(\psi(t_1),e_i)|^2\\
&\leq \frac{1}{t_3-t_1}\big\|\nabla \psi(t_1)\big\|^2_{L^2(\Omega)}.
\end{split}
\end{equation}
This, together with (\ref{1.71}), leads to
\begin{equation}\label{1.7}
I^2_1\leq \frac{1}{(t_3-t_1)^2}\|\nabla\psi(t_1)\|^2_{L^2(\Omega)}.
\end{equation}
Since $\psi=0$ and $\Delta\psi=0$ on $\partial\Omega\times(0,T)$, applying the elliptic regularity and the Poincar\'{e} inequality, we obtain that
\begin{equation*}
\begin{split}
  \|\nabla\psi(t_2)\|^2_{H^2(B_r)}\leq \|\psi(t_2)\|^2_{H^3(\Omega)}\leq N(\Omega)\|\Delta\psi(t_2)\|^2_{H^1(\Omega)}
   \leq N(\Omega) \|\nabla\Delta\psi(t_2)\|^2_{L^2(\Omega)}.
   \end{split}
  \end{equation*}
By similar arguments as those to derive  (\ref{1.71})---(\ref{1.7}), we can verify that
\begin{equation*}
\begin{split}
 \|\nabla\Delta\psi(t_2)\|^2_{L^2(\Omega)}
 \leq \frac{1}{(t_2-t_1)^2}\|\nabla\psi(t_1)\|^2_{L^2(\Omega)}.
   \end{split}
\end{equation*}
Now, the above two estimates yield that
\begin{equation*}\label{1.8}
 I_3^2=\|\nabla\psi(t_2)\|^2_{H^2(B_r)}
  \leq \frac{N(\Omega)}{(t_2-t_1)^2}\|\nabla\psi(t_1)\|^2_{L^2(\Omega)}.
\end{equation*}
Along with \eqref{1.9} and \eqref{1.7}, this leads to
  \begin{equation*}
  \begin{split}
    &\|\nabla\psi(t_2)\|_{L^2(\Omega)}\\ &\leq(t_3-t_1)^{-\frac 12}\left(Ne^{\frac{N}{t_2-t_3}}\|\nabla\psi(t_2)
    \|_{L^2(B_r)}\right)^{\frac {1-\alpha}{2}}
    \|\nabla\psi(t_1)\|^{\frac 12}_{L^2(\Omega)}\|\nabla\psi(t_3)\|^{\frac{\alpha}{2}}
    _{L^2(\Omega)}\\
    &\leq(t_3-t_1)^{-\frac 12}\left(Ne^{\frac{N}{t_2-t_3}}\|\nabla\psi(t_2)
    \|_{L^2(B_r)}\right)^{\frac {1-\alpha}{2}}
    \|\nabla\psi(t_1)\|^{\frac {1+\alpha}{2}}_{L^2(\Omega)},\\
\end{split}
\end{equation*}
when  $0< t_1<t_3<t_2\leq T$.
Choosing $t_3=\frac{t_1+t_2}{2}$ in last estimates and recalling that
$B_r\subset\omega$, we get at once that
  \begin{equation}\label{stands}
    \|\nabla\psi(t_2)\|_{L^2(\Omega)}\leq \left(Ne^{\frac{N}{t_2-t_1}}\|
    \nabla\psi(t_2)\|_{L^2(\omega)}\right)^{\frac {1-\alpha}{2}}\|\nabla\psi(t_1)\|^{\frac {1+\alpha}{2}}_{L^2(\Omega)}.
  \end{equation}
  The desired estimate \eqref{threeball}  stands  if we replace $\frac{1-\alpha}{2}$ by $\alpha$ in \eqref{stands}.

  Finally, the unique continuation in the second part of this lemma is an immediate consequence of the estimate \eqref{threeball} and Lemma~\ref{prop1}.
\end{proof}

\begin{proof}[\bf Proof of Theorem~\ref{uniquecontinuation}]
Arbitrarily fix a $\mathbf{u}_0\in L^2_{\sigma}(\Omega)$. Since $\Omega$ is simply connected, according to Lemma~\ref{rottheorem}, there exists a unique stream function $\psi_0\in H_0^1(\Omega)$ such that $\mathbf{curl}~\psi_0=\mathbf{u}_0$.
 Let $\psi$ be the solution to Equation \eqref{stream} with the aforementioned  $\psi_0$. By Proposition $\ref{wellpose}$,   $\mathbf{u}:=\curl\psi$ is the unique solution of Equations \eqref{stokes} with the initial condition $\mathbf{u}(0)=\mathbf{u}_0$.
From Lemma $\ref{phungwang}$, it follows that $\psi$ holds
the estimate \eqref{threeball}.
Since
$$\|\nabla\psi(t)\|_{L^2(O)}=
  \|\curl\psi(t)\|_{L^2(O)}=\|\mathbf{u}(t)\|_{L^2(O)}\;\;\mbox{ for each}\;\;t\in [0,T],
  $$
  where $O$ is either $\Omega$ or $\omega$, and because  $\mathbf{u}_0$ was arbitrarily taken from $L^2_{\sigma}(\Omega)$,
  the desired estimate \eqref{c*} follows  from (\ref{threeball}) at once.

   Consequently, if $\|\mathbf{u}(t)\|_{L^2(\omega)}=0$ for some $t>0$, then by the estimate \eqref{threeball} and Lemma~\ref{phungwang}, we find that $\mathbf{u}\equiv0$.

   Finally, by making use of  the new strategy
  in \cite{unpublished1}, we can derive the observability estimate \eqref{c88} from \eqref{c*}.
\end{proof}


\section{Applications }
\setcounter{equation}{0}
$\;\;\;\;$Let $T>0$ and $\Omega\subset\mathbb{R}^2$ be a bounded and simply connected domain with a $C^3$ boundary $\partial\Omega$.
Consider the following controlled Stokes equations:
\begin{equation}\label{control}
  \begin{cases}
\mathbf{u}_t-\Delta\mathbf{u}-\nabla p=\mathbf{f}\;\;\;\;&\text{in}\;\;\Omega\times(0,T),\\
\nabla\cdot\mathbf{u}=0\;\;\;\;\;\;\;&\text{in}\;\;\Omega\times(0,T),\\
\rot \mathbf{u}=0,\;\;\mathbf{u}\cdot \mathbf{n}=0\;\;\;\;&\text{on}\;\;\partial\Omega\times(0,T),\\
\mathbf{u}(\cdot,0)=\mathbf{u}_0(\cdot) \;\;\;\;&\text{in}\;\;\Omega.
\end{cases}
\end{equation}
In what follows,  $H^{-1}_\sigma(\Omega)$  stands for the dual of $H^1_\sigma(\Omega)$, and
$\langle\cdot,\cdot\rangle$  the scalar product between $H^{-1}_\sigma(\Omega)$ and $H^1_\sigma(\Omega)$. We first define the weak solution to equation \eqref{control}.

\begin{definition}
For each $\mathbf{f}\in L^2(0,T;L^2(\Omega))$ and each $\mathbf{u}_0\in L^2_{\sigma}(\Omega)$, $\mathbf{u}$ is called a weak solution of Equations $(\ref{control})$, if
\begin{equation*}
  \mathbf{u}\in  C([0,T];L^2_{\sigma}(\Omega))\cap L^2(0,T;H^1_{\sigma}(\Omega))
  ,\quad\mathbf{u}_t\in L^2(0,T;H^{-1}_{\sigma}(\Omega))~\text{with}~ \mathbf{u}(\cdot,0)=\mathbf{u}_0
\end{equation*}
and
\begin{equation} \label{weakformula}
       \int^T_0\langle\mathbf{u}_t,\mathbf{v}\rangle ~dt +\int^T_0(\rot\mathbf{u},\rot\mathbf{v})~dt=
    \int^T_0(\mathbf{f},\mathbf{v})~dt,~\forall\,\mathbf{v}\in L^2(0,T;H^1_{\sigma}(\Omega)).
  \end{equation}
\end{definition}
\begin{remark}
{\it
It can be verified that the solution $\mathbf{u}$ obtained in Proposition $\ref{wellpose}$ is  a weak solution of Equations $(\ref{control})$ with $\mathbf{f}=0$.}
\end{remark}
The following proposition is concerned with the existence and uniqueness of the weak solution of Equations \eqref{control} and we leave its proof  in Appendix of this paper.
\begin{prop}\label{wellposeness}
For each $\mathbf{f}\in L^2(0,T;L^2(\Omega))$ and each $\mathbf{u}_0\in L^2_{\sigma}(\Omega)$, Equations $(\ref{control})$ has a unique weak solution.
\end{prop}

In what follows, we will denote by $\mathbf{u}(\cdot; \mathbf{f})$ the weak solution to Equation (\ref{control}) corresponding to the exterior force $ \mathbf{f}$, when the initial datum $\mathbf{u}_0$ is given.
\subsection{Null Controllability of the Stokes Equations}
$\;\;\;\;$In this subsection, we will
show that Theorem~\ref{uniquecontinuation} implies the null controllability of Stokes equations with control restricted over $\omega\times E$,
where $\omega\subset\Omega$ is a nonempty open subset, and $E$ is a  subset of positive measure in $(0,T)$.
Denoted by $\mathbf{u}(\cdot\,;\mathbf{f}\chi_{\omega}\chi_{E})$
the unique weak solution to Equations \eqref{control} corresponding to the control $\mathbf{f}$ restricted on the subset $\omega\times E$.
\begin{corollary}\label{theoremcan3}
For each $\mathbf{u}_0\in L^2_{\sigma}(\Omega)$, there exists a control $\mathbf{f}\in L^{\infty}(0,T;L^2(\Omega))$, with
\begin{equation}\label{controlcost}
\|\mathbf{f}\|_{L^\infty(0,T;L^2(\Omega))}\leq N\|\mathbf{u}_0\|_{L^2(\Omega)}, \;\;N=N(\Omega,\omega,E,T),
\end{equation}
such that $\mathbf{u}(T;\mathbf{f}\chi_{\omega}\chi_{E})=0$.
\end{corollary}

Before giving the proof of Corollary~\ref{theoremcan3}, we first state an interpolation lemma quoted from \cite[pp. 260-261]{TemamNavier}.
\begin{lemma}
Let $V, H, V'$ be three Hilbert spaces, each space included in the following one: $V\subset H\equiv H'\subset V'$, $V'$ being the dual of $V$. If a function $u\in L^2(0,T;V)$ and its derivative $u'\in L^2(0,T;V')$, then $u$ is almost everywhere equal to a continuous function from $[0,T]$ into $H$ and we have the following equality, which holds in the scalar distribution sense on $(0,T)$:
  \begin{equation}\label{derivative}
    \frac{d}{dt}\|u\|_{H}^2=2\langle u',u\rangle_{V',V}.
  \end{equation}
\end{lemma}
\begin{remark} The equality $(\ref{derivative})$ is meaningful
since $\langle u'(t),u(t)\rangle_{V',V}$ is integrable on $(0,T)$. Using $(\ref{derivative})$ for $u+v$, we have
  \begin{equation}\label{derivative1}
    \frac{d}{dt}(u, v)_H = \langle u',v\rangle_{V',V}+\langle v',u\rangle_{V',V},
  \end{equation}
  for all  $u,v\in L^2(0,T;V)$ with $u',v'\in L^2(0,T;V')$.
\end{remark}

\begin{proof}[\bf Proof of Corollary \ref{theoremcan3}]
We first introduce the following adjoint system of Equations \eqref{control}:
 \begin{equation}\label{adjoint}
  \begin{cases}
\mathbf{v}_t+\Delta\mathbf{v}+\nabla q=0\;\;\;\;&\text{in}\;\;\Omega\times(0,T),\\
\nabla\cdot\mathbf{v}=0\;\;\;\;\;\;\;&\text{in}\;\;\Omega\times(0,T),\\
\rot \mathbf{v}=0,\;\;\mathbf{v}\cdot \mathbf{n}=0\;\;\;\;&\text{on}\;\;\partial\Omega\times(0,T),\\
\mathbf{v}(\cdot,T)=\mathbf{v}_T(\cdot)\;\;\;\;&\text{in}\;\;\Omega.
\end{cases}
\end{equation}
For each $\mathbf{v}_T\in L^2_{\sigma}(\Omega)$, according to Proposition $\ref{wellpose}$, Equations \eqref{adjoint} has a unique solution $\mathbf{v}\in L^2(0,T;H^1_\sigma(\Omega))$.

By Theorem~\ref{uniquecontinuation}, there exists a positive constant
$N=N(\Omega,\omega,E,T)$ such that
  \begin{equation}\label{inequality}
  \|\mathbf{v}(0)\|_{L^2(\Omega)}\leq N\int^T_{0}\chi_{E}\|\mathbf{v}(t)\|_{L^2(\omega)}\,dt.
\end{equation}
Now, set
\begin{equation*}
  \mathbf{X}\triangleq\left\{\overline{\mathbf{v}}=
  \mathbf{v}\chi_{\omega}\chi_{E}
  :~\mathbf{v}\;\text{ solves Equations}\;\; \eqref{adjoint}
  \;\;\text{with}\;\;\mathbf{v}_T(\cdot)\in L^2_\sigma(\Omega) \right\}.
\end{equation*}
Let $\mathbf{X}$ be endowed with the norm of  $L^1(0,T;L^2(\Omega))$. Clearly, it is a subspace of $L^1(0,T;L^2(\Omega))$.  Next, we define a linear functional $\mathbf{F}:\mathbf{X}\rightarrow\mathbb{R}$ by
\begin{equation*}
  \mathbf{F}(\overline{\mathbf{v}})=-( \mathbf{v}(0),\mathbf{u}_0).
\end{equation*}
Note that $\mathbf{F}$ is well defined. In fact,  if $\overline{\mathbf{v}}_1=\overline{\mathbf{v}}_2$, then it follows  from Theorem \ref{uniquecontinuation} that $\mathbf{v}_1=\mathbf{v}_2$.  By \eqref{inequality}, we have  that
\begin{equation*}
  |\mathbf{F}(\overline{\mathbf{v}})|  \leq \| \mathbf{v}(0)\|_{L^2(\Omega)}\|\mathbf{u}_0\|_{L^2(\Omega)}
     \leq N\|\mathbf{u}_0\|_{L^2(\Omega)}
     \int_E\|\mathbf{v}(t)\|_{L^2
     (\omega)}\,dt.
\end{equation*}
Hence, $\mathbf{F}$ is a bounded linear functional on $\mathbf{X}$. By the Hahn-Banach theorem,  $\mathbf{F}$ can be extended to a bounded linear functional on $L^1(0,T;L^2(\Omega))$. Using the Riesz representation theorem, we get that there exists $\mathbf{f}\in L^{\infty}(0,T;L^2(\Omega))$ such that
\begin{equation*}
  \mathbf{F}(\mathbf{h})=
\int^T_0\int_{\Omega}\mathbf{f}(x,t)\cdot\mathbf{h}(x,t)
\,dxdt
\end{equation*}
and
\begin{equation*}
|\mathbf{F}(\mathbf{h})|\leq N\|\mathbf{u}_0\|_{L^2(\Omega)}\|\mathbf{h}\|_{L^1(0,T;\mathbf
{L}^2(\Omega))},\;\;\text{for all}\;\; \mathbf{h}\in L^1(0,T;\mathbf
{L}^2(\Omega))
\end{equation*}
In particular, for each $\overline{\mathbf{v}}\in\mathbf{X}$, we have that
\begin{equation}\label{long}
  -(\mathbf{v}(0),\mathbf{u}_0)=\int^T_0(\mathbf{f},
  \overline{\mathbf{v}})\,dt
  =\int^T_0(\mathbf{f} \chi_{\omega}\chi_E,\mathbf{v})\,dt.
\end{equation}

We next verify that $\mathbf{u}(T;\mathbf{f}\chi_{\omega}\chi_{E})=0$.  To serve this purpose, we first use \eqref{derivative1} to get
\begin{equation}\label{ccccaaaa}
  \left(\mathbf{u}(T;\mathbf{f}\chi_{\omega}\chi_{E}),\mathbf{v}_T\right)-
  \left(\mathbf{u}(0),\mathbf{v}(0)\right)
     =\int^T_0\langle\mathbf{u}_t,\mathbf{v}\rangle+
  \langle\mathbf{v}_t,\mathbf{u}\rangle\, dt.
\end{equation}
Then, by \eqref{weakformula}, we obtain that
\begin{equation*}
\begin{split}
  \int^T_0\langle\mathbf{u}_t,\mathbf{v}\rangle&+
  \langle\mathbf{v}_t,\mathbf{u}\rangle\, dt
  =\int^T_0(\rot\mathbf{v},\rot\mathbf{u})\,dt+\int^T_0
  (\mathbf{f}\chi_{\omega}\chi_E,\mathbf{v})\,dt\\
  &- \int^T_0(\rot\mathbf{u},\rot\mathbf{v})\,dt
  =\int^T_0( \mathbf{f}\chi_{\omega}\chi_E,\mathbf{v})\,dt.
  \end{split}
\end{equation*}
This, along with \eqref{long} and \eqref{ccccaaaa}, leads to
\begin{equation*}
  (\mathbf{u}(T;\mathbf{f}\chi_{\omega}\chi_{E}),\mathbf{v}_T)=0,~\text{for all}~ \mathbf{v}_T\in L^2_{\sigma}(\Omega).
  \end{equation*}
Hence $\mathbf{u}(T;\mathbf{f}\chi_{\omega}\chi_{E})=0$. This completes the proof.
\end{proof}

\subsection{The Bang-bang Property of the Time and Norm Optimal Control Problem}
$\;\;\;\;$In the sequel, we make use of Corollary~\ref{theoremcan3} to get the bang-bang property for the  minimal norm and minimal time control problems for Stokes equations.
We begin with introducing these problems.
Let $\omega\subset \Omega$ be a nonempty
open subset, and let $\textbf{u}_0\in L^2_\sigma(\Omega)\setminus\{0\}$.
 For each $T>0$, define the following  control constraint set:
$$\mathcal{F}_{T}=\left\{\mathbf{f}\in L^\infty (0,T;L^2(\Omega)):\;
\mathbf{u}(T;\mathbf{f}\chi_\omega)=0 \right\}.$$
 According to Corollary~\ref{theoremcan3}, the set $\mathcal{F}_T$ is nonempty. Consider the minimal norm control problem:
$$(NP)_T:\;\;\;\;M_{T}\equiv \min\left\{\|\mathbf{f}\|_{L^\infty (0,T;L^2(\Omega))}:\;\;
\mathbf{f}\in\mathcal{F}_{T}\right\}.$$
Since $\mathcal{F}_{T}$ is not empty, it follows from  the standard arguments (see, e.g., \cite{unpublished1}) that Problem $(NP)_{T}$ has solutions. A solution of this problem is called  a minimal norm control.

Now, one can use the same methods as those in  \cite{unpublished1}  to prove the following consequence of Corollary~\ref{theoremcan3}:

\begin{corollary}\label{cor1} Problem $(NP)_T$ has the bang-bang property: any minimal norm control $\mathbf{f}$ satisfies that $\|\mathbf{f}(t)\chi_{\omega}\|_{L^2(\Omega)}=M_{T}$ for a.e. $t\in (0,T)$. Consequently, this problem has a
unique minimal norm control in $L^\infty(0,T;L^2(\omega))$.
\end{corollary}

\vskip 8pt
Next, for each $M>0$, we define the following  control constraint set:
$$\mathcal{U}_M=\{\mathbf{g}\in L^\infty(\mathbb{R}^+;L^2(\Omega)):
\;\;\|\mathbf{g}(t)\|_{L^2(\Omega)}\leq M\;\;\text{for a.e.}\;\;t\in\mathbb{R}^+\}.$$
Consider the minimal time control problem:
$$(TP)^M:\;\; T_M\equiv \min_{\mathbf{g}\in\mathcal{U}_M}
\left\{t>0:\;\;\textbf{u}(t;\mathbf{g}\chi_{\omega})=0\right\},$$
where $\textbf{u}(\cdot\,;\mathbf{g}\chi_{\omega})$ is the solution
to
\begin{equation}\label{control2}
  \begin{cases}
\mathbf{u}_t-\Delta\mathbf{u}-\nabla p=\mathbf{g}\chi_{\omega}\;\;\;\;&\text{in}\;\;\Omega
\times(0,\mathbb{R}^+),\\
\nabla\cdot\mathbf{u}=0\;\;\;\;\;\;\;&\text{in}\;\;\Omega
\times(0,\mathbb{R}^+),\\
\rot \mathbf{u}=0,\;\;\mathbf{u}\cdot \mathbf{n}=0\;\;\;\;&\text{on}\;\;\partial\Omega\times(0,\mathbb{R}^+),\\
\mathbf{u}(\cdot,0)=\mathbf{u}_0(\cdot) \;\;\;\;&\text{in}\;\;\Omega.
\end{cases}
\end{equation}
According to Corollary~\ref{theoremcan3} and the energy decay property of solutions to homogenous Stokes equations, we
see that $\mathcal{U}_M$ is nonempty. By the standard arguments (see, e.g., \cite[Lemma 3.2]{MR2342129}), Problem $(TP)^M$  has solutions. A solution of this problem is called a minimal time control.

One can follow the similar way as that in  \cite{MR2421326}, \cite{LQ}
 or \cite{wangzhang} to show the following consequence of Corollary~\ref{theoremcan3}:
\begin{corollary}\label{cor2}
Problem $(TP)^M$ has the bang-bang property: any minimal time control $\mathbf{f}$ satisfies that $\|\mathbf{f}(t)\chi_{\omega}\|_{L^2(\Omega)}=M$ for a.e. $t\in (0, T_M)$. Consequently, this problem has a
unique minimal time control in $L^\infty(0,T_M;L^2(\omega))$.
\end{corollary}

\begin{proof}
Let $\mathbf{f}^*$ be the minimal control and $t^*$ be the minimal time for the problem $(TP)^M$. Suppose by contradiction that there were $E\subset(0,t^*)$ of positive measure  and $\varepsilon>0$
such that
\begin{equation*}
\|\mathbf{f}^*(t)\chi_{\omega}\|_{L^2(\Omega)}\leq M-\varepsilon\;\;\;\text{for a.e.}\;\;t\in E.
\end{equation*}
It suffices to show that there are a positive number $\delta$ with $\delta< t^*$ and a control $\mathbf{g}_{\delta}\in\mathcal{U}_{M}$ such that the following holds:
\begin{equation}\label{cancan2}
\mathbf{u}(t^*-\delta; \chi_{\omega}\mathbf{g}_{\delta})=0.
\end{equation}
This means that $t^*$ could not be the optimal time,
which leads to a contradiction. Indeed, from Corollary~\ref{theoremcan3},
it follows  that for each positive number $\delta$
sufficiently small, there exists a control $\mathbf{f}_{\delta}$,
with the estimate
\begin{equation*}
\|\mathbf{f}_{\delta}(t)\|_{L^\infty(0,t^*-\delta;\,L^2(\Omega))}
\leq\frac{\varepsilon}{2},
\end{equation*}
 such that
\begin{equation*}
\mathbf{z}^\delta(t^*-\delta)=0,
\end{equation*}
where $\mathbf{z}^\delta(\cdot)$ is the weak solution to the following controlled equation:
\begin{equation*}
 \begin{cases}
\mathbf{z}^\delta_t-\Delta \mathbf{z}^\delta-\nabla p^\delta =\mathbf{f}_{\delta}\chi_{\omega}\chi_{E_\delta}
\;\;\;&\text{in}\;\;\Omega\times(0,t^*-\delta),\\
\nabla\cdot \mathbf{z}^\delta=0\;\;\;&\text{in}\;\;\Omega\times(0,t^*-\delta),\\
\rot \mathbf{z}^\delta=0,\;\;\mathbf{z}^\delta\cdot \mathbf{n}=0\;\;&\text{on}\;\;\partial\Omega\times(0,t^*-\delta),\\
\mathbf{z}^\delta(0)=\mathbf{u}_{0}-\mathbf{u}(\delta; \mathbf{f}^*\chi_{\omega}),\;\;&\text{in}\;\;\Omega.
   \end{cases}
 \end{equation*}
Here $E_{\delta}\triangleq\{t>0:\;t+\delta\in E\}$
(Clearly, $|E_\delta|\geq|E|-\delta$). Set
\begin{equation*}
\mathbf{g}_{\delta}(t)=
\begin{cases}
\mathbf{f}^*(t+\delta)+\chi_{E_\delta}(t)\mathbf{f}_{\delta}(t),\;\;
&t\in(0,t^*-\delta),\\
0,\;\;\;\;\;&t\in[t^*-\delta,\infty).
\end{cases}
\end{equation*}
Clearly, $\mathbf{g}_\delta\in\mathcal{U}_{M}$.
It is easy to check that if we choose $\mathbf{g}_{\delta}$
as the control in Equations~\eqref{control2}, the equality (\ref{cancan2}) will
be valid. The  uniqueness of the optimal control follows directly from the bang-bang
property and the parallelogram identity (see, e.g., \cite{MR2421326}).
\end{proof}

\section{Appendix}
\setcounter{equation}{0}
\begin{proof}[\bf Proof of Lemma~\ref{biharmoniclemma}]
First, let $\phi=\Delta\psi$.
  It's clear that $\phi$ solves the following heat equation:
\begin{equation*} \label{heat}
  \begin{cases}
\phi_t-\Delta\phi=0\;\;\;\;&\text{in}\;\;  \Omega\times(0,T),\\
 \phi=0~~~~~~~~\;\;\;\;\;&\text{on}\;\;\partial\Omega\times(0,T),\\
\phi(\cdot,0)=\Delta \psi_0(\cdot) \;\;\;\;&\text{in}\;\;\Omega.
   \end{cases}
 \end{equation*}
 Since $ \Delta\psi_0\in H^{-1}(\Omega)$, by the classical results of heat equation, it has  a unique solution
\begin{equation*}\label{interpolation}
\phi\in C([0,T];H^{-1}(\Omega))\cap C^1((0,T]; H^{-1}(\Omega))\cap C((0,T];H^1_0(\Omega))\;\;\text{with}\;\;\phi(0)=\Delta \psi_0.
\end{equation*}
Consequently, $\Delta\psi\in C((0,T]; H^1_0(\Omega))$.
Next, for each $t_1,t_2\in [0,T]$, we consider the following elliptic equation:
\begin{equation*}
    \begin{cases}
 \Delta\big(\psi(t_1)-\psi(t_2)\big)=\phi(t_1)-\phi(t_2)
 \;\;\;\;&\text{in}\;\;  \Omega,\\
 \psi(t_1)-\psi(t_2)=0 \;\;\;\;&\text{on}\;\;\p\Omega,
   \end{cases}
 \end{equation*}
 the elliptic regularity  indicates that
\begin{equation*}
\|\psi(t_1)-\psi(t_2)\|_{H_0^1(\Omega)}\leq N(\Omega) \|\phi(t_1)-\phi(t_2)\|_{H^{-1}(\Omega)},
\end{equation*}
which, together with $\phi\in
  C([0,T];H^{-1}(\Omega))$, implies that $\psi\in C([0,T]; H^1_0(\Omega))$.  Similarly, $\psi\in C^1((0,T]; H^1_0(\Omega))\cap C((0,T]; H^3(\Omega))$. The uniqueness is obvious.
\end{proof}

\begin{proof}[\bf Proof of Lemma~\ref{rottheorem}]
For each $\mathbf{u}\in L^2_{\sigma}(\Omega)$, there exists a sequence of $\big\{\mathbf{u}_n\}_{n\geq 1}\subset\{\mathbf{v}\in C^{\infty}_0(\Omega;\mathbb{R}^2)~:~\nabla\cdot\mathbf{v}=0\big\}$ (see, e.g., \cite[pp. 13--16]{TemamNavier}) such that
\begin{equation}\label{approximation}
  \mathbf{u}_n\rightarrow\mathbf{u}\quad\text{in}~ L^2(\Omega).
\end{equation}
For each $n\geq 1$, let $\psi_n\in H^1_0(\Omega)$ be the unique solution to
\begin{equation}\label{elliptic}
  \begin{cases}
       -\Delta\psi_n=\rot\mathbf{u}_n\;\;&\text{in}\;\;\Omega,\\
       \psi_n=0\;\;&\text{on}\;\;\partial\Omega.
     \end{cases}
  \end{equation}
  From classical results on elliptic equations, it follows that
\begin{equation*}\label{estimate1}
    \|\psi_n\|_{H^1_0(\Omega)}\leq N(\Omega)\|\mathbf{u}_n\|_{L^2(\Omega)}.
\end{equation*}
  Because $\psi_n=0$ on $\partial \Omega$, we have
  $$\curl\psi_n\cdot\mathbf{n}=\frac{\p\psi_n}{\p\mathbf{\tau}}=0\;\;\mbox{over} \;\;\partial\Omega.
  $$
   This implies that $\curl\psi_n\in L^2_{\sigma}(\Omega)$.
  Observe that Equation \eqref{elliptic} can also be written as $\rot(\curl\psi_n-\mathbf{u}_n)=0$. Since $\Omega$ is simply
  connected, we get that  $\curl\psi_n-\mathbf{u}_n\in\ker(\rot)\cap L^2_{\sigma}(\Omega)=0$ (see, for instance, \cite[Remark 2.2, pp. 32]{femstokes}). In fact, $\ker(\rot)\cap L^2_{\sigma}(\Omega)$ is isomorphic to the first space of cohomology $\mathbb{H}^1(\Omega;\mathbb{R})$, which is zero for a simply connected domain (we refer the readers to \cite[Appendix 1, Remark 1.1, pp. 463]{TemamNavier} for a detailed argument for such topics). Hence, $\curl\psi_n=\mathbf{u}_n$.

Since
\begin{equation*}
  \begin{cases}
       -\Delta(\psi_n-\psi_m)=\rot(\mathbf{u}_n-\mathbf{u}_m)
       \;\;&\text{in}\;\;\Omega,\\
       \psi_n-\psi_m=0\;\;&\text{on}\;\;\partial\Omega,
     \end{cases}
  \end{equation*}
 we have that
    \begin{equation*}
    \|\psi_n-\psi_m\|_{H^1_0(\Omega)}\leq N(\Omega)\|\mathbf{u}_n-\mathbf{u}_m\|_{L^2(\Omega)}.
 \end{equation*}
This, along with \eqref{approximation}, indicates that $\{\psi_n\}_{n\geq 1}$ is a Cauchy sequence in $H^1_0(\Omega)$ and hence converges to some $\psi\in H^1_0(\Omega)$. In conclusion,
  \begin{equation*}
  \begin{cases}
       \curl\psi_n\rightarrow\curl\psi~&\text{in}\;~L^2(\Omega),\\
       \curl\psi_n=\mathbf{u}_n\rightarrow\mathbf{u}~
       &\text{in}~\;L^2(\Omega).
     \end{cases}
  \end{equation*}
  Therefore, $\curl\psi=\mathbf{u}$.
\end{proof}
\vskip 10pt

\begin{proof}[\bf Proof of Proposition \ref{wellposeness}]
We first define on $H^1_{\sigma}(\Omega)\times H^1_{\sigma}(\Omega)$ a bilinear functional for a.e. $t\in(0,T)$ by
\begin{equation}\label{bilinear}
  a(t;\mathbf{u},\mathbf{v})=(\rot\mathbf{u},
  \rot\mathbf{v}).
\end{equation}
Since $\Omega$ is simply connected, we have that (see, e.g., \cite[Remark 3.5, pp. 45]{femstokes})
$$\|\rot\mathbf{v}\|_{L^2(\Omega)}^2\geq N(\Omega)\|\mathbf{v}\|^2_{H^1(\Omega)},\;\;\mathbf{v}\in H^1_\sigma(\Omega).$$
It can be verified that
\begin{equation*}
\begin{cases}
 &|a(t;\mathbf{u},\mathbf{v})|\leq M\|\mathbf{u}\|_{H^1(\Omega)}\|\mathbf{v}\|_{H^1(\Omega)},~ \forall~t\in [0,T],~\forall\;\; \mathbf{u},\mathbf{v}\in H^1_{\sigma}(\Omega).\\
 &a(t;\mathbf{v},\mathbf{v})\geq \Lambda\|\mathbf{v}\|^2_{H^1(\Omega)},~ \forall~t\in [0,T],~\forall\; \mathbf{v}\in H^1_{\sigma}(\Omega),\;\;\text{where}\;\;\Lambda>0.
\end{cases}
\end{equation*}
By the Lions theorem (see, e.g., \cite[Chapitre X, Commentaires, pp. 218]{MR697382}), we obtain that there exists
  \begin{equation*}
  \mathbf{u}\in  C([0,T];L^2_{\sigma}(\Omega))\cap L^2(0,T;H^1_{\sigma}(\Omega))
  ,\quad\mathbf{u}_t\in L^2(0,T;H^{-1}_{\sigma}(\Omega))~\text{with}~ \mathbf{u}(\cdot,0)=\mathbf{u}_0,
\end{equation*}
such that
\begin{equation*}
\langle\mathbf{u}_t,\mathbf{v}\rangle=- (\rot\mathbf{u}(t),\rot\mathbf{v})+
    (\mathbf{f}(t),\mathbf{v}), ~
    \text{for a.e.} ~t\in [0,T]\;\;\text{and for all}\;\;\mathbf{v}\in H^1_\sigma(\Omega).
\end{equation*}
In particular, for any $\mathbf{v}\in L^2(0,T;H^1_{\sigma}(\Omega))$,
\begin{equation*}\label{pointwise}
       \langle\mathbf{u}_t,\mathbf{v}(t)\rangle=- (\rot\mathbf{u}(t),\rot\mathbf{v}(t))+
    (\mathbf{f}(t),\mathbf{v}(t) ), ~
    \text{for a.e.} ~t\in [0,T].
\end{equation*}
  Since the right hand side of above equality is integrable on $(0,T)$, we
  immediately get the desired equality \eqref{weakformula}. By
  a standard energy method, it follows that
\begin{equation*}
 \max_{s\in[0,T]}\|\mathbf{u}(s)\|^2_{L^2(\Omega)}\leq \|\mathbf{u}_0\|^2_{L^2(\Omega)}
 +N(\Omega)\|\mathbf{f}\|^2_{L^2(0,T;L^2(\Omega))},
\end{equation*}
and this completes the proof.
\end{proof}
\begin{remark}\label{m-t}
{\it The bilinear functional  defined by \eqref{bilinear} induces a
self-adjoint maximal monotone operator $A$ in $L^2_\sigma(\Omega)$
with $$D(A)=\{\mathbf{u}\in H^2(\Omega)\cap H^1_\sigma(\Omega); \rot{\mathbf{u}}|_{\partial\Omega}=0\}$$ and
$A\mathbf{u}=-P\Delta \mathbf{u}$ when $\mathbf{u}\in D(A)$.
Here $P$ is the Helmholtz projection operator.
}
\end{remark}

\end{document}